\documentclass[12pt,psamsfonts]{amsart}

\usepackage{amsmath,amssymb}
\usepackage{graphicx}
\usepackage[dvips]{psfrag}

\newtheorem{theorem}{Theorem}
\newtheorem{proposition}[theorem]{Proposition}

\newtheorem {remark}[theorem]{Remark}
\newtheorem{definition}[theorem]{Definition}

\title[Stable piecewise  polynomial vector fields]{
Stable piecewise polynomial vector fields}
\author[Claudio Pessoa and Jorge Sotomayor]{}

\thanks{The first author is supported by FAPESP-–Brazil Project 2011/13152-8 and by Programa
Primeiros Projetos—-PROPe/UNESP. %%
The second author is fellow of CNPq and participates in the
Project Fapesp 2008/02841-4.}

  \subjclass{Primary 34C35, 58F09; Secondary 34D30}
   \keywords{structural stability, piecewise vector fields, infinity,
   compactification}

\begin{document}
 \maketitle

\centerline{\scshape  Claudio Pessoa}
\medskip

{\footnotesize \centerline{Universidade Estadual Paulista,
UNESP--IBILCE}\centerline{Av. Cristov\~ao Colombo, 2265}
\centerline{ 15.054--000, S. J. Rio Preto, SP, Brasil }
\centerline{\email{pessoa@ibilce.unesp.br}}}

\medskip

\centerline{\scshape Jorge Sotomayor}
\medskip

{\footnotesize \centerline{ Instituto de Matem\'atica e Estat\'\i
stica, Universidade de S\~ao Paulo} \centerline{Rua do Mat\~ao
1010, Cidade Universit\'aria} \centerline{05.508-090, S\~ao Paulo,
SP, Brasil } \centerline{\email{sotp@ime.usp.br}}}

\medskip

\bigskip

\begin{quote}{\normalfont\fontsize{8}{10}\selectfont
{\bfseries Abstract.} Consider in $\mathbb R^2$ the semi-planes
$N=\{y>0\}$ and $S=\{y<0\}$ having as common boundary the straight
line $D=\{y=0\}$. In $N$ and $S$ are defined polynomial vector
fields $X$ and $Y$, respectively, leading to a discontinuous piecewise
 polynomial vector field $Z=(X,Y)$. This work pursues the
stability and the transition analysis of solutions of $Z$ between
$N$ and $S$, started by Filippov (1988) and Kozlova (1984) and
reformulated by Sotomayor--Teixeira (1995) in terms of the
regularization method. This method consists in analyzing a one
parameter family of continuous vector fields $Z_{\epsilon}$,
defined by averaging $X$ and $Y$. This family approaches $Z$ when
the parameter goes to zero. The results of Sotomayor--Teixeira and
Sotomayor--Machado (2002) providing conditions on $(X,Y)$ for the
regularized vector fields to be structurally stable on planar
compact connected regions are extended to discontinuous piecewise
polynomial vector fields on $\mathbb R^2$.  Pertinent genericity
results for vector fields satisfying the above  stability
conditions are also extended to the present case. A procedure for
the study of discontinuous piecewise vector fields at infinity
through a compactification is proposed here.
\par}
\end{quote}

\section{Introduction}
\label{sec:01}

One of the most accomplished stability theories for dynamical
systems is that of Andronov--Pontryagin \cite {AGLM} and Peixoto
\cite {P} for $C^1$ vector fields in the plane and on surfaces.
Elements of this theory provide characterization and genericity
results for structurally stable vector fields. Extensions of this
theory to the class of discontinuous, piecewise smooth, vector
fields have been provided by Filippov \cite {F} and Kozlova \cite
{K}.   The need for such an extended theory goes back to Andronov
et al. \cite {ACH}.

In \cite {F}, Filippov defined the rules (revisited below)  for
the transition of the orbits  crossing the line $D$ of
discontinuity which separates two regions $N$ and $S$ on which the
field, given  respectively by  $X$ and $Y$, is smooth. He also
prescribed when the orbit  slides  along $D$. This leads to an
orbit structure that is not always a flow on the surface obtained
gluing $N$  and $S$ along $D$. The work of Kozlova \cite {K, F}
pursues the setting established  by Filippov.

In \cite {ST},  Sotomayor and Teixeira  developed the {\it
regularization method}, taking as domain the sphere $S^2$ and the
equation as the discontinuity line $D$. This method consists in
defining a one parameter family of continuous vector fields that,
when the parameter goes to zero, approaches the discontinuous one.
To this end, a transition function $\varphi$ is used to average $X$
and $Y$ in order to get the family of continuous vector fields.
Sotomayor and Teixeira provided conditions on $Z=(X,Y)$, which imply
that the regularized vector fields are in the class of
Andronov--Pontryagin \cite {AGLM} and Peixoto \cite {P} for $C^1$
vector fields and consequently are structurally stable. Moreover,
Sotomayor and Machado \cite{MS} applied the method outlined above to
the case of a compact planar region $M$, with a smooth border
$\partial M$ and having as discontinuity line either a segment with
extremes on $\partial M$ or a closed curve disjoint of $\partial M$.
The conditions given in \cite {ST} are extended to this case and
their genericity, not discussed in \cite {ST}, is established.

Other developments in this direction can be found in
Garcia--So\-to\-ma\-yor \cite{GS}, where piecewise linear vector
fields are studied and in Buzzi--da Silva--Teixeira \cite{BST},
where the method of singular perturbations is used to study certain
discontinuous piecewise smooth  vector fields. For interesting
examples in applied subjects of discontinuous systems the reader is
addressed to \cite{ACH} and \cite{BCK}.

In this paper we deal with discontinuous piecewise vector fields $Z$
defined by a pair $(X,Y)$, where $X$ and $Y$ are polynomial vector
fields in the plane.

A {\it polynomial vector field $X$ in $\mathbb R^{2}$} is a vector
field of the form
\[
X=P(x,y)\frac{\partial}{\partial x}+Q(x,y)\frac{\partial}{\partial
y},
\]
where $P$ and $Q$ are polynomials in the variables $x$ and $y$
with real coefficients. We define the {\it degree} of the
polynomial vector field $X$ as $\max\{\deg P, \deg Q\}$. We can
write $P(x,y)=\sum a_{ij}x^iy^j$ and $Q(x,y)=\sum b_{ij}x^iy^j$,
$0\leq i+j\leq m$. Hence $X$ has degree $\leq m$. The
$l=(m+1)(m+2)$ real numbers $\{a_{i,j}, b_{ij}\}$ are called the
{\it coefficients of} $X$. The space of these vector fields,
endowed with the structure of affine $\mathbb R^l$-space where $X$
is identified with the $l$-tuple
$(a_{00},a_{10},\ldots,a_{0m},b_{00},\ldots,$ $b_{0m})$ of its
coefficients, is denoted by $\chi_m$.

Let $f:\mathbb R^2 {\rightarrow} {\mathbb R}$  be  the function
$f(x,y)=y$. In what follows we use the following notation: $D =
f^{-1}(0)$, $N = f^{-1}(0,\infty)$ and $S = f^{-1}(-\infty,0)$.

Let $\Omega_m$ be the space of vector fields $Z=(X,Y)$ defined by:
$$ Z(q) = \left\{ \begin{array}{c}             X(q)
\hspace{0.3cm}   \mbox{if}   \hspace{0.3cm}   f(q)   \geq    0,    \\
Y(q)   \hspace{0.3cm}   \mbox{if}   \hspace{0.3cm}   f(q)   \leq
0,
\end{array}\right.$$
where $X,Y \in \chi_m$ and $\deg X =\deg Y=m$. We write $Z
=(X,Y)$, which will be allowed
 to be bi-valued at points of $D$. In general the degrees of $X$ and $Y$ can be different,
 but in the present study,  to simplify the notation and some computations,
 we take them to be equal.

The Poincar\'e compactification of $X\in \chi_m$ is defined to be
the unique analytic vector field $\mathcal P (X)$ tangent to the
sphere $S^2=\{x^2+y^2+z^2=1\}$ whose restriction to the northern
hemisphere $S^2_+=\{S^2:z>0\}$ is given by $z^{m-1}\wp^*(X)$,
where $\wp$ is the central projection from $\mathbb R^2$ to
$S^2_+$, defined by $\wp(u,v)=(u,v,1)/\sqrt{u^2+v^2+1}$. See
\cite{G} for a verification of the uniqueness and analyticity of
$\mathcal P (X)$.

Through the Poincar\'e compactification, the discontinuous
piecewise polynomial vector field $Z=(X,Y)$ induces a
discontinuous piecewise analytic
 vector field tangent to $S^2$, with $S^1$
invariant, defined by $\mathcal P (Z)=(\mathcal P(X),\mathcal
P(Y))$.  Notice that, for $\mathcal P (Z)$ restricted to the
northern hemisphere, the function $f$ becomes $f(x,y,z)=y$ with
$(x,y,z)\in S^2_+$, the set of discontinuity is given by
$D=\{S^2:y=0\}$ and the semi-planes $N$ and $S$ become the
semi-hemispheres $N=\{S^2:z>0\;\mbox{and}\; y>0\}$ and
$S=\{S^2:z>0\;\mbox{and}\; y<0\}$, respectively. Thus, $\mathcal
P(Z)$ can be used to study the global structure of the orbits of
$Z$.

By a {\it transition function}  we mean a $C^{\infty}$ function
$\varphi: \mathbb R \to \mathbb R$ such that: $\varphi (t)  =  0$
if $t \leq -1$, $\varphi (t) = 1$  if  $t  \geq 1$ and
$\varphi^{\prime} (t) >  0$  if  $t  \in (-1,1)$.

\begin{definition} \label{1.1} $\frac{}{}$
The ${\varphi}_{\epsilon}$-{\it compactification  of} $Z = (X,Y) \in
\Omega_m$ is   the   one   parameter   family of $C^{\infty}$ vector
fields $\mathcal P(Z)_{\epsilon}$ in $S^2$ given by
$$\mathcal P( Z)_{\epsilon}(q) = (1 - \varphi_{\epsilon}(f(q)))\mathcal P(Y)(q) + \varphi_{\epsilon}(f(q))\mathcal P(X)(q),$$
where $\varphi_{\epsilon}(t)=\varphi(\frac {t} {\epsilon})$. \
\end{definition}

Denote  by  $\chi^r(S^2,S^1)$ the space of $C^{r}$ vector fields
on $S^2$, $r\geq 1$, such that $S^1$ is invariant by the flow of
the vector fields.

\begin{definition}
\label{def:01} $X\in \chi^r(S^2,S^1)$ is said to be structurally
stable if there is a neighborhood $V$ of $X$ and a map
$h:V\rightarrow Hom(S^2,S^1)$ (homeomorphisms of $S^2$ which
preserve $S^1$) such that $h_X=Id$ and $h_Y$ maps orbits of
$\mathcal P(X)$ onto orbits of $\mathcal P(Y)$, for every $Y\in
V$.
\end{definition}

\begin{definition}
\label{def:02} We call $\Sigma^r(S^2,S^1)$ the subset of
$\chi^r(S^2,S^1)$ of vector fields that have all their singularities
hyperbolic, all their periodic orbits hyperbolic and do not have
saddle connections in $S^2$ unless they are contained in $S^1$.
\end{definition}
We have that the elements of $\Sigma^r(S^2,S^1)$ are structurally
stable in the sense of definition \ref{def:01}.

In the next sections of this paper we will extend to the
 case of  discontinuous
piecewise polynomial vector fields in $\mathbb R^2$
  the study performed in \cite{ST,MS} for piecewise smooth vector fields.
  To this end we will give
sufficient conditions on $Z=(X,Y)\in \Omega_m$ which determine the
structural stability of its $\varphi_\epsilon$-compactification
$\mathcal P(Z)_\epsilon$ (Definition \ref{1.1}), for any transition
function $\varphi$ and small $\epsilon$. More precisely in Section
\ref{sec:03} will be defined a set $G_m$ (Definition \ref{4.1}) of
discontinuous piecewise polynomial vector fields that satisfy
sufficient conditions, reminiscent to those which define
$\Sigma^r(S^2,S^1)$, in order to have a structurally stable
$\varphi_\epsilon$-compactification. In Section \ref{sec:04}, the
genericity of $G_m$ will be established. A preliminary analysis of
relevant local aspects of discontinuous piecewise polynomial vector
fields is developed in Section \ref{sec:02}. There is studied the
effect of $\varphi_\epsilon$-compactification on singular points,
closed orbits  and polytrajectoris (Definition \ref{2.4}) in
$\mathbb R^2$ and in $S^1$ and on saddle separatrices.

\section{$\varphi_\epsilon$-Compactification of Singular Points,
Closed  and    Saddle Separatrix Poly-Trajectories } \label{sec:02}

In this section, using the notations, definitions and results of
\cite{ST, MS}, we define the regular and singular points of $Z$
(resp. $\mathcal P(Z)$), the closed {\it poly-trajectories} and
then we study the effects of the
$\varphi_\epsilon$-compactification on vector fields around these
points and poly-trajectories. The main goal here is to determine
the conditions for the $\varphi_\epsilon$-compactification to have
only regular points, hyperbolic singularities and hyperbolic
closed orbits.

\subsection{Regular and Singular Points}

 Given any $Z = (X,Y) \in \Omega_m$, following Filippov
terminology (as \cite {F}), we distinguish the following arcs in
$D$:
\begin{itemize}
\item Sewing  Arc  $(SW)$:  characterized  by  $(Xf)(Yf)  >  0$
(see Figure~\ref{fig:01} (a)). \item Escaping Arc $(ES)$: given by
the inequalities $Xf>0$ and $Yf<0$ (see Figure~\ref{fig:01} (b)).
\item Sliding Arc $(SL)$: given by the  inequalities $Xf<0$ and
$Yf>0$ (see Figure~\ref{fig:01} (c)).
\end{itemize}

As usual, here and in what follows, $Xf$ will denote the
derivative of the function $f$  in the direction of the vector
$X$,  i.e.,  $Xf=\langle \nabla f, X\rangle$.

\begin{figure}[h!]
\begin{center}
\psfrag{A}[l][B]{\tiny$(a)$} \psfrag{B}[l][B]{\tiny$(b)$}
\psfrag{C}[l][B]{\tiny$(c)$} \psfrag{N}[l][B]{\tiny$N$}
\psfrag{S}[l][B]{\tiny$S$} \psfrag{D}[l][B]{\tiny$D$}
\includegraphics[height=1.2in,width=2.5in]{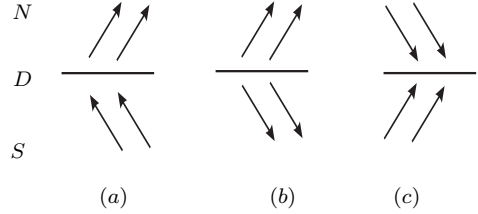}
\end{center}
\caption{Arcs on $D$} \label{fig:01}
\end{figure}

On the arcs $ES$ and $SL$ we define the {\it Filippov vector field}
$F_{Z}$ associated  to $Z = (X,Y)$, as follows: if $p \in SL$ or
$ES$, then $F_{Z}(p)$  denotes  the vector in the cone spanned by
$X(p)$  and $Y(p)$ that is
 tangent to $D$,
see Figure~\ref{fig:02}.

\begin{figure}[h!]
\begin{center}
\psfrag{X}[l][B]{\tiny$X(p)$} \psfrag{Y}[l][B]{\tiny$Y(p)$}
\psfrag{P}[l][B]{\tiny$p$} \psfrag{N}[l][B]{\tiny$N$}
\psfrag{S}[l][B]{\tiny$S$} \psfrag{D}[l][B]{\tiny$D$}
\psfrag{F}[l][B]{\tiny$F_{Z}(p)$}
\includegraphics[height=1in,width=1.8in]{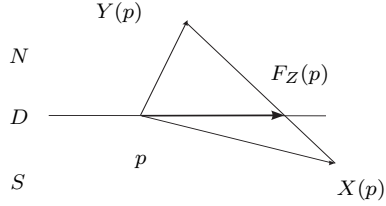}
\end{center}
\caption{Filippov vector field} \label{fig:02}
\end{figure}

\begin{definition}\label{2.1}
A  point  $p  \in  D$  is  called a {\it D-regular point of} $Z$ if
one    of       the    following    conditions  holds:
\begin{enumerate}
\item  $Xf(p).Yf(p) > 0$. This   means  that  $p  \in  SW$; \item
$Xf(p).Yf(p) < 0$ but $det[X,Y](p) \neq 0$. This means that $p$
belongs either to $SL$ or $ES$ and it is not a singular point of
$F_{Z}$ (see Figure~\ref{fig:03}).
\end{enumerate}
\end{definition}

Now, we define the notion of hyperbolicity for the singular points
of $F_{Z}$.

\begin{definition}\label{2.2}
A point $p \in D$ is called a {\it singular point of} $F_{Z}$
if\break $Xf(p).Yf(p) < 0$ and $det[X,Y](p) = 0$. If we have
$d(det[X,Y]|_{D})(p) \neq 0$, then $p$ is called a {\it hyperbolic
singular point of} $F_{Z}$. Here $d(det[X,Y]|_{D})(p)$ denote the
derivative of $det[X,Y]|_{D}$ at point $p$.

Let $p \in D$ be a hyperbolic singular point of $F_{Z}$. The point
$p$ is called a {\it saddle} if $p \in SL$ and
$d(det[X,Y]|_{D})(p) > 0$ or $p \in ES$ and $d(det[X,Y]|_{D})(p) <
0$. The point $p$ is called a {\it node} if $p \in SL$ and
$d(det[X,Y]|_{D})(p) < 0$ or if $p \in ES$ and
$d(det[X,Y]|_{D})(p) > 0$ (see Figure~\ref{fig:04}).
\end{definition}

In the next definition we extend the notion of hyperbolic singular
point, located  in D,  for $Z$.

\begin{definition}\label{2.3}
A point $p \in D$ is  an  {\it elementary D-singular point of} $Z
= (X,Y)$ if one of  the  following conditions is satisfied:
\begin{enumerate}
\item The point $p$ is a  {\it fold  point of} $Z = (X,Y)$. This
means that: either $p$ is a fold point of $X$: $Yf(p) \neq 0,
Xf(p) = 0$ and $X^2f(p) \neq 0$; or $p$ is a fold point of $Y$:
$Xf(p)  \neq 0, Yf(p) = 0 $ and $Y^2f(p) \neq 0$ (see
Figure~\ref{fig:05}); \item The point $p$  is  a hyperbolic
singular point of $F_{Z}$.
\end{enumerate}
\end{definition}

The definitions above can be reformulated in a similar way in the
case of  discontinuous piecewise analytic vector field $\mathcal
P(Z)$ in $S^2$.

To determine the behavior of singular points and periodic orbits of
$\mathcal P(Z)$ we will obtain an expression of $\mathcal P(Z)$ in
polar coordinates. Take coordinates $(\theta,\rho)$, $2\pi$-periodic
in $\theta$, defined by the covering map from $(-1,1)\times \mathbb
R$ onto $S^2\setminus \{(0,0,\pm 1\}$, given by
$(\theta,\rho)\mapsto (x,y,z)=(1+\rho^2)^{-1/2}(\cos \theta,\sin
\theta, \rho)$.

The expression for $z^{m-1}\wp^*(X)$, $X=(P,Q)\in\chi_m$, in these
coordinates is
\[
(1+\rho^2)^{(1-m)/2}\left[\left(\sum \rho^i
A_{m-i}(\theta)\right)\frac{\partial}{\partial
\theta}-\rho\left(\sum \rho^iR_{m-i}(\theta)
\right)\frac{\partial}{\partial \rho}\right],
\]
where $i=0,1,\ldots,m$ and
\[
A_k(\theta)=A_k(X,\theta)=Q_k(\cos \theta,\sin
\theta)\cos\theta-P_k(\cos \theta,\sin \theta)\sin\theta, \]
\[
R_k(\theta)=R_k(X,\theta)=P_k(\cos \theta,\sin
\theta)\cos\theta+Q_k(\cos \theta,\sin \theta)\sin\theta,
\]
with $P_k=\sum a_{ij}x^iy^j$, $Q_k=\sum b_{ij}x^iy^j$, $i+j=k$.
Now, we perform a change in the time variable to remove the factor
$(1+\rho^2)^{(1-m)/2}$ and to obtain a vector field defined in the
whole plane $(\theta, \rho)$, i.e. we have the vector field
\begin{equation}
\label{eq:01} \left(\sum \rho^i
A_{m-i}(\theta)\right)\frac{\partial}{\partial
\theta}-\rho\left(\sum \rho^iR_{m-i}(\theta)
\right)\frac{\partial}{\partial \rho},
\end{equation}
with $i=0,1,\ldots,m$.
 Note that we also can obtain \eqref{eq:01} directly from $X=(P,Q)$ introducing in the plane $(x,y)$ the
change of variables $x=\cos\theta/\rho$, $y= \sin \theta/\rho$.
Moreover, the axis $\theta$, i.e. $\{(\theta,\rho):\rho=0\}$, is
invariant by \eqref{eq:01} and corresponds to the points at
 infinity of $\mathbb R ^2$.
 Therefore, to study the behavior of solutions of $\mathcal
P(Z)$, $Z=(X,Y)\in \Omega_m$ with $X=(P_1,Q_1)$ and $Y=(P_2,Q_2)$,
is equivalent by \eqref{eq:01} to study the discontinuous
piecewise trigonometric vector field
\begin{equation}
\label{eq:03} \left\{ \begin{array}{l}             \left(\sum
\rho^i A_{1,m-i}(\theta),-\rho\sum \rho^iR_{1,m-i}(\theta)
\right), \hspace{0.3cm} \mbox{if} \hspace{0.3cm} \theta\in
[0,\pi], \rho\geq 0,    \\ \\ \left(\sum \rho^i
A_{2,m-i}(\theta),-\rho\sum \rho^iR_{2,m-i}(\theta) \right),
\hspace{0.3cm} \mbox{if}
 \hspace{0.3cm} \theta\in [\pi,2\pi], \rho\geq
0,
\end{array}\right.
\end{equation}
with $i=0,1,\ldots,m$, where $A_{1,k}(\theta)=A_{k}(X,\theta)$,
$A_{2,k}(\theta)=A_{k}(Y,\theta)$,
$R_{1,k}(\theta)=R_{k}(X,\theta)$ and
$R_{2,k}(\theta)=R_{k}(Y,\theta)$.

We remark that $S^1\cap D=\{(\pm 1,0,0)\}$. Hence, if $p\in
S^1\cap D$ is not a singular point of $\mathcal P(X)$ and
$\mathcal P (Y)$ then, as $S^1$ is invariant by $\mathcal P(Z)$
and so by $\mathcal P(Z)_\epsilon$ (Definition \ref{1.1}) , it
follows that $p$ is a point of sewing arc $SW$ or $p$ is a
singular point of the Filippov vector field $F_{\mathcal P(Z)}$.

Suppose that $(1,0,0)$ is a singular point of $F_{\mathcal P(Z)}$.
This point corresponds to the point $(0,0)$ in the chart
$(\theta,\rho)$ and in this chart $D=\{(0,\rho): \rho \geq
0\}\cup\{(\pi,\rho): \rho\geq 0\}\cup\{(2\pi,\rho): \rho \geq
0\}$. Therefore, by \eqref{eq:03}, it follows that $\det[\mathcal
P(X),\mathcal P(Y)]|_{(0,\rho)}=$
\[
-\rho\left[\sum \rho^i A_{1,m-i}(0)\sum \rho^i R_{2,m-i}(0)-\sum
\rho^i R_{1,m-i}(0)\sum \rho^i A_{2,m-i}(0)\right],
\]
and so $\displaystyle\frac{d}{d\rho}\left(\det[\mathcal
P(X),\mathcal P(Y)]|_{(0,\rho)}\right)(0)=$
\[
R_{1,m}(0)A_{2,m}(0)-A_{1,m}(0)R_{2,m}(0).
\]
Hence, $(1,0,0)$ is a hyperbolic singular point of $F_{\mathcal
P(Z)}$ if and only if
\begin{equation}
\label{eq:04}
P_{1,m}(1,0)Q_{2,m}(1,0)-Q_{1,m}(1,0)P_{2,m}(1,0)\neq 0,
\end{equation}
where $P_{k,m}$ and $Q_{k,m}$ are the homogeneous parts of degree
$m$ of $P_k$ and $Q_k$, respectively, $k=1,2$.

Now, in a similar way we have that if $(-1,0,0)$ is a singularity
of $F_{\mathcal P(Z)}$ then it is hyperbolic if \eqref{eq:04}
holds.

The proofs of the propositions below are analogous to those proofs
of the respective propositions (Proposition $6$ page $230$,
Proposition $8$ page $231$ and Proposition $9$ pag. $233$)
established in \cite{MS}.

\begin{proposition}\label{prop1.1} Let  $p \in S^2_+\cup S^1$ be a $D$-regular point  of  $\mathcal P(Z)$ with
$Z=(X,Y) \in \Omega_m$. Then, given a transition function
$\varphi$, there  exists  a neighborhood $V$ of $p$ and
$\epsilon_{0} > 0$ such that for $0 < \epsilon \leq \epsilon_{0}$,
$\mathcal P(Z)_{\epsilon}$ has no singular points in $V$ (see
Figure~\ref{fig:03}).
\end{proposition}

\begin{figure}[h!]
\begin{center}
\includegraphics[height=1.5in,width=1.5in]{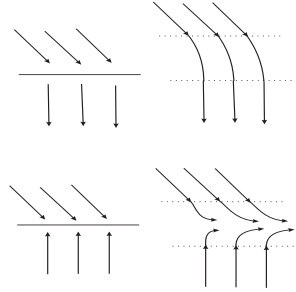}
\end{center}
\caption{D-regular points and their
$\varphi_\epsilon$-compactification} \label{fig:03}
\end{figure}

\begin{proposition}\label{prop1.5}
Given $Z = (X,Y) \in \Omega_m$, let $p$ be a hyperbolic singular
point of $F_{\mathcal P(Z)}$. Then, given a transition function
$\varphi$, there is  a neighborhood $V$ of $p$ in $S^2_+\cup S^1$
and $\epsilon_{0} > 0$ such   that   for     $0 < \epsilon \leq
\epsilon_{0}$, $\mathcal P(Z)_{\epsilon}$ has near $p$  a unique
singular point which is a hyperbolic saddle or a hyperbolic node
(see Figure~\ref{fig:04}).
\end{proposition}

\begin{figure}[h!]
\begin{center}
\includegraphics[height=1.5in,width=1.5in]{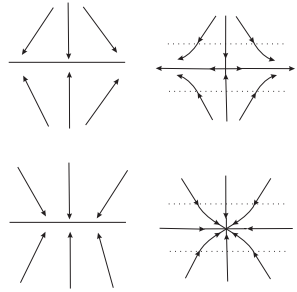}
\end{center}
\caption{$D$-singular points and their regularizations}
\label{fig:04}
\end{figure}

\begin{proposition}\label{prop1.6} Let $p$ be a fold point of  $\mathcal P(Z)$ with $Z  =  (X,Y)\in \Omega_m$.
Then, given a transition function $\varphi$, there is a
neighborhood $V$ of $p$ and $\epsilon_{0} > 0$ such  that  for  $
0 < \epsilon  \leq \epsilon_{0}, \mathcal P(Z)_{\epsilon}$ has no
singular points in $V$ (see Figure~\ref{fig:05}).
\end{proposition}

\begin{figure}[h!]
\begin{center}
\includegraphics[height=1.5in,width=1.5in]{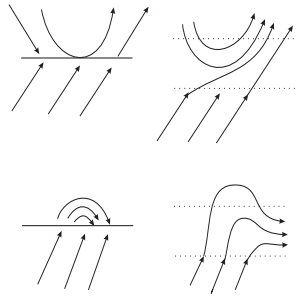}
\end{center}
\caption{Fold points and their
$\varphi_\epsilon$-compactification} \label{fig:05}
\end{figure}

\subsection{Closed and Saddle Connections  Poly-Trajectories}
\begin{definition}\label{2.4}
A continuous curve $\gamma$ consisting of regular trajectory arcs
of $X$ and/or of $Y$ and/or of $F_{Z}$ is called a {\it
poly-trajectory} if:
\begin{enumerate}
\item  $\gamma$ has arcs of at least two fields among $X, Y$ and
$F_{Z}$, or consists of a single arc of $F_{Z}$;

\item  the transition between arcs of $X$ and $Y$ happens on the
sewing arc; \item the transition between arcs of $X$ or $Y$ and
$F_{Z}$ occurs at fold points or regular points of the sliding or
the escaping arcs, preserving the sense of the arcs (see
Figure~\ref{fig:08}).
\end{enumerate}
\end{definition}

\begin{figure}[h!]
\begin{center}
\psfrag{D}[l][B]{\tiny$D$} \psfrag{N}[l][B]{\tiny$N$}
\psfrag{S}[l][B]{\tiny$S$}
\includegraphics[height=0.7in,width=3in]{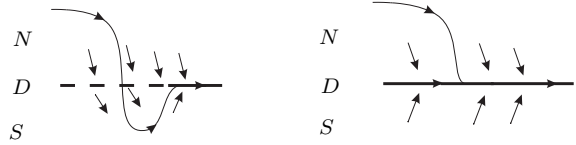}
\end{center}
\caption{Poly-Trajectories} \label{fig:08}
\end{figure}

Now we define saddle connections on $Z$.

\begin{definition} $\frac{}{}$

\begin{itemize} \item[a)] A separatrix of $Z$ is a trajectory of $X$, $Y$ or $F_Z$ such
that its $\alpha$ or $\omega$-limit sets are  saddle points of $X$,
$Y$ or $F_Z$.

\item[b)] A double separatrix of $Z$ is a trajectory of $X$, $Y$
or $F_Z$ such that their $\alpha$ and $\omega$-limit sets are
saddles or a separatrix of $X$ (resp. $Y$) that meets $D$ at a
saddle of $F_Z$.

\item[c)] A saddle connection of $Z$ is a double separatrix or a
poly-trajectory that contains a double separatrix or two
separatrices (see Figure~\ref{fig:07}).
\end{itemize}
\end{definition}

\begin{figure}[h!]
\begin{center}
\psfrag{D}[l][B]{\tiny$D$} \psfrag{N}[l][B]{\tiny$N$}
\psfrag{S}[l][B]{\tiny$S$}
\includegraphics[height=1in,width=3in]{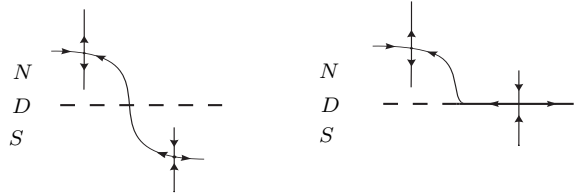}
\end{center}
\caption{Saddle Connections} \label{fig:07}
\end{figure}

Now we define closed trajectories of $Z$ that have points or arcs
of $D$.

\begin{definition}\label{2.5}
Let $\gamma$ be a closed poly-trajectory of $Z = (X,Y)$.
\begin{enumerate}
\item $\gamma$ is called a {\it closed poly-trajectory of type 1}
if $\gamma$ meets $D$ only at sewing points; \item $\gamma$ is
called a {\it closed poly-trajectory of type 3} if it has at least
one fold point and one sliding or escaping arc of $Z$ (see
Figure~\ref{fig:06}).
\end{enumerate}
\end{definition}

Next definition extends the notion of hyperbolic orbits for closed
poly-trajectory of $Z$.

\begin{definition}\label{2.6}
Let $\gamma$ be a closed poly-trajectory of $Z = (X,Y)\in
\Omega_m$. It is called {\it elementary} if one of the cases below
holds:
\begin{enumerate}
\item $\gamma$ is of type 1 and has a first return map $\eta$ with
$\eta^{\prime} \neq 1$;  \item $\gamma$ is of type 3 and all arcs
of $F_{Z}$ are sliding or all are escaping.
\end{enumerate}
\end{definition}

The definitions above can be reformulated in similar way for the
discontinuous piecewise analytic vector field $\mathcal P(Z)$ in
$S^2$.

Now, we will study the stability of $S^1$ when it is a closed
poly-trajectory of $\mathcal P(Z)$. Note that in this case $S^1$
is necessarily of type $1$. Moreover $m$ is odd, otherwise always
there are singular points of $\mathcal P(Z)$ in $S^1$. We will
need the following result that can be found in \cite{AGLM, S}.

\begin{proposition}
\label{pro:01} Let $X$ be a $C^1$ planar vector field. Given a point
$p_0\in \mathbb R^2$, denote by $\phi (t,p_0)$ the orbit of $X$ such
that $\phi (0,p_0)=p_0$ and by $p_1$ the point $\phi(T_0,p_0)$. Let
$\Sigma_0$ and $\Sigma_1$ be transversal sections of $X$ at the
points $p_0$ and $p_1$, respectively. If $\sigma:I\rightarrow
\mathbb R^2$ and $\hat{\sigma}:\hat{I}\rightarrow \mathbb R^2$ are
the respective parameterizations of $\Sigma_0$ and $\Sigma_1$ with
$\sigma(s_0)=p_0$ and $\hat{\sigma}(\hat{s}_0)=p_1$, then the
derivative of the transition map $\Pi:\Sigma_0\rightarrow \Sigma_1$
at the point $p_0$, defined by the flow of $X$, is given by
\[
\displaystyle
\Pi'(p_0)=\frac{\det\left(\displaystyle\begin{array}{c}X(p_0)\\\sigma'(s_0)\end{array}\right)}
{\det\left(\displaystyle\begin{array}{c} X(p_1)\\
\hat{\sigma}'(\hat{s}_0)\end{array}\right)}
e^{\displaystyle\int_0^{T_0}\mbox{\rm div} X(\phi(t,p_0))dt}.
\]
\end{proposition}

Denote by $\tilde{Z}=(\tilde{X},\tilde{Y})$ the discontinuous
piecewise polynomial  vector field which gives rise to system
\eqref{eq:03}. In the plane $(\theta, \rho)$ the points $p_0=(0,0)$,
$p_2=(2\pi,0)$, correspond to the point $(1,0,0)$ of $S^1\cap D$,
and $p_1=(\pi,0)$ corresponds to the other point $(-1,0,0)$. As
$S^1$ is a closed poly-trajectory of type $1$, we can take the
following transversal sections $\Sigma_0=\{(0,\rho): 0\leq
\rho\leq\delta_0\}$, $\Sigma_1=\{(\pi,\rho): 0\leq
\rho\leq\delta_0\}$ and $\Sigma_2=\{(2\pi,\rho): 0\leq
\rho\leq\delta_0\}$ of \eqref{eq:03} with $\delta_0$ small enough.
Hence, we define the following transition maps
$\Pi_1:\Sigma_0\rightarrow \Sigma_1$, $\Pi_2:\Sigma_1\rightarrow
\Sigma_2$ and obtain the Poincar\'e map $\Pi$ of $\mathcal P(Z)$
associated to $S^1$ in the coordinates $(\theta,\rho)$, given by
$\Pi=\Pi_2\circ\Pi_1$. We have that $
\Pi'(p_0)=\Pi_2'(\Pi_1(p_0))\Pi_1'(p_0)=\Pi_2'(p_1)\Pi_1'(p_0)$.
Thus, by Proposition \ref{pro:01} and expression \eqref{eq:03}, it
follows that
\[
\begin{array}{lcl}
\Pi_1'(p_0) & =& \displaystyle
-\frac{Q_{1,m}(1,0)}{Q_{1,m}(-1,0)}e^{\displaystyle\int^{T_1}_0\mbox{div}\tilde{X}(\theta(t),0)dt}
\\ & = &
e^{\displaystyle\int^{T_1}_0\left(-R_{1,m}(\theta(t))+\frac{dA_{1,m}}{d\theta}(\theta(t))\right)dt},
\end{array}
\]
with $\dot{\theta}(t)=A_{1,m}(\theta(t))$, $\theta (0)=0$ and
$\theta(T_1)=\pi$. Therefore,
\[
\Pi_1'(p_0)  = \displaystyle
\frac{A_{1,m}(\pi)}{A_{1,m}(0)}e^{\displaystyle-\int^{\pi}_0\frac{R_{1,m}(\theta)}{A_{1,m}(\theta)}
d\theta}  =
e^{\displaystyle-\int^{\pi}_0\frac{R_{1,m}(\theta)}{A_{1,m}(\theta)}d\theta}.
\]
Analogously, we have
\[
\Pi_2'(p_1)
=e^{\displaystyle-\int^{2\pi}_{\pi}\frac{R_{2,m}(\theta)}{A_{2,m}(\theta)}
d\theta}  =
e^{\displaystyle-\int^{\pi}_0\frac{R_{2,m}(\theta)}{A_{2,m}(\theta)}d\theta}.
\]
Hence,
\[
\Pi'(p_0)=e^{\displaystyle-\int^{\pi}_0\left(\frac{R_{1,m}(\theta)}{A_{1,m}(\theta)}+\frac{R_{2,m}(\theta)}{A_{2,m}(\theta)}\right)d\theta}.
\]
Note that we have performed the computations above supposing that
$S^1$ is oriented in the counterclockwise sense.

Now we can state the following proposition.
\begin{proposition}
Suppose that $\mathcal P(Z)$, $Z\in\Omega_m$, with $m$ odd, does not
have singular points in $S^1$. Then $S^1$ is a closed
poly-trajectory of type $1$ and the derivative of the Poincar\'e map
associated to a transversal section at the point $p_0\in S^1\cap D$
is given by
\[
\Pi'(p_0)=e^{\sigma\mu}=e^{\displaystyle\sigma\int^{\pi}_0\left(\frac{R_{1,m}(\theta)}{A_{1,m}(\theta)}+\frac{R_{2,m}(\theta)}{A_{2,m}(\theta)}
\right)d\theta},
\]
where $\sigma=-1$, if $S^1$ is oriented in the counterclockwise
sense, and $\sigma=1$, otherwise. Moreover, $S^1$ is an attractor if
$\sigma\mu<0$ and a repeller if $\sigma\mu>0$.
\end{proposition}

We conclude that $S^1$ is an elementary closed poly-trajectory if
and only if
\begin{equation}
\label{eq:05}
\int^{\pi}_0\left(\frac{R_{1,m}(\theta)}{A_{1,m}(\theta)}+\frac{R_{2,m}(\theta)}{A_{2,m}(\theta)}\right)d\theta\neq
0.
\end{equation}

The proof of the proposition below is analogous to the proof of
Proposition $13$, page 234, established in \cite{MS}.

\begin{proposition}\label{prop1.7} Let $\gamma$ be an elementary closed poly-trajectory of
$\mathcal P(Z)$ with $Z  =  (X,Y)\in \Omega_m$.  Then, given a
transition function $\varphi$, there is a neighborhood $V$ of
$\gamma$ and $\epsilon_{0} > 0$  such  that  for   $ 0 < \epsilon
\leq \epsilon_{0}$,   $\mathcal P(Z)_{\epsilon}$ has only one
periodic orbit in $V$, and this orbit is hyperbolic (see
Figure~\ref{fig:06}).
\end{proposition}

\begin{figure}[ptb]
\begin{center}
\psfrag{T}[l][B]{\tiny type $3$} \psfrag{T1}[l][B]{\tiny type $1$}
 \psfrag{N}[l][B]{\tiny$N$}
\psfrag{S}[l][B]{\tiny$S$} \psfrag{D}[l][B]{\tiny$D$}
\includegraphics[height=2in,width=2in]{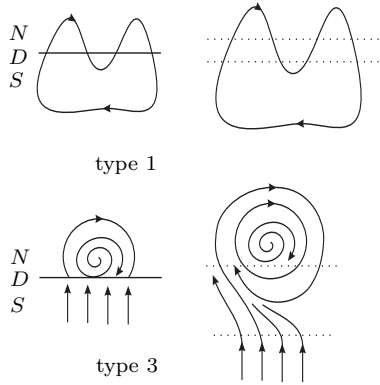}
\end{center}
\caption{Closed poly-trajectories and their
$\varphi_\epsilon$-compactification} \label{fig:06}
\end{figure}

\section{Piecewise Polynomial Vector Fields with Structurally Stable $\varphi_\epsilon$-Compactification}
\label{sec:03}

In this section we define a set   $G_m$ of discontinuous piecewise
polynomial  vector fields whose elements, $Z$, have structurally
stable $\varphi_\epsilon$-compactification $\mathcal
P(Z)_\epsilon$ (Definition \ref{1.1}), for any transition function
$\varphi$ and small $\epsilon$.

The notion of structural stability in $\chi_m$ is defined in
similar way as in $\chi^r(S^2,S^1)$ (see Definition \ref{def:01}).
Denote by  $\Sigma_m$  the set of $X\in \chi_m$ that are
structurally stable.

\begin{definition} \label{4.0}
We call $\mathcal S_m$ the set of all polynomial vector fields $X
\in \chi_m$ for which $\mathcal P(X)$ satisfies the following
conditions:
\begin{itemize}
\item[(1)] all its singular points are hyperbolic; \item[(2)] all
its periodic orbits are hyperbolic; \item[(3)] it does not have
saddle connections in $S^2$ unless they are contained in $S^1$.
\end{itemize}
\end{definition}

We have that $\mathcal S_m\subset \Sigma_m$ and it is an open and
dense set of $\chi_m$. However, it is  an {\it unsolved  problem}
to prove (or disprove) that $\mathcal S_m=\Sigma_m$. See \cite{S1}
for more details.

\begin{remark}\label{rmk:abuso}
By extension of the notation in definition \ref{4.0}, we will
write in what follows $ X|_{N} \in \mathcal S_m$ and $Y|_{S} \in
\mathcal S_m$  to mean that  conditions (1), (2) and (3) in this
definition hold for $ X|_{N}$ and $Y|_{S}$.
\end{remark}

\begin{definition} \label{4.1}
Write $G_m = G_m(1) \cap G_m(2) \cap G_m(3)$, where:
\begin{enumerate}
\item $G_m(1) = \{Z = (X,Y) \in \Omega_m :X|_{N}$ and $Y|_{S} \in
\mathcal S_m$; each $D$-singularity of $\mathcal P(Z)$ is elementary
$\}$. \item $G_m (2) = \{Z = (X,Y) \in \Omega_m:X|_{N}$ and $Y|_{S}
\in \mathcal S_m$; each closed poly-trajectory of $\mathcal P(Z)$ is
elementary $\}$. \item $G_m(3) = \{Z = (X,Y) \in \Omega_m:X|_{N}$
and $Y|_{S} \in \mathcal S_m$; $\mathcal P(Z)$ does not have saddle
connections in $S^2$ unless they are contained in $S^1\}$.
\end{enumerate}
\end{definition}

\begin{proposition}\label{prop2.4}
Let $Z = (X,Y) \in G_m(1)$. Then, given a transition function
$\varphi$, there is an $\epsilon_{0} > 0$ such that for $ 0 <
\epsilon  \leq \epsilon_{0}$, $\mathcal P(Z)_{\epsilon}$ has only
hyperbolic singularities in $S^2$.
\end{proposition}
\begin{proof}
As $X|_{N}$ and $Y|_{S} \in \mathcal S_m$, it remains to prove
that the singularities that appear due to the
$\varphi_\epsilon$-compactification process are hyperbolic.
Indeed, let $p$ be a point of $D$, then $p$ can be a $D$-regular
point, a hyperbolic singularity of $F_{\mathcal P(Z)}$ or a fold.
For each case, there is a proposition that guarantees the
existence a number $\epsilon_{0} > 0$ such that, for each
$\epsilon \in (0, \epsilon_{0}]$, $\mathcal P(Z)_{\epsilon}$ has
no singularities near $p$ (Propositions \ref{prop1.1},
\ref{prop1.6}) or has a unique hyperbolic singularity (Proposition
\ref{prop1.5}). The union of these neighborhoods cover $D$, and,
as $D$ is compact, there is a sub covering made by a finite number
of these neighborhoods. Then, we can chose $\epsilon_{0}$ as the
smallest $\epsilon_{0}$ associated to these neighborhoods.
\end{proof}
\begin{proposition}\label{prop4.8}
Let $Z = (X,Y) \in G_m(2) \cap G_m(3)$. Then, given a transition
function $\varphi$, there is an $\epsilon_{0} > 0$ such that  for
$0 <  \epsilon  \leq \epsilon_{0}$, $\mathcal P(Z)_{\epsilon}$ has
only hyperbolic periodic orbits in $S^2$.
\end{proposition}
\begin{proof}
As $X|_{N}$ and $Y|_{S} \in \mathcal S_m$, all their periodic
orbits are hyperbolic, so it remains to prove that the same occurs
to the periodic orbits that appear by the
$\varphi_\epsilon$-compactification process. Let $\gamma$ be an
elementary closed poly-trajectory of $\mathcal P(Z)$. Then, by
Proposition \ref{prop1.7}, there is an $\epsilon_{0}
> 0$  such  that  for every  $ 0 < \epsilon \leq \epsilon_{0}$,
$\mathcal P(Z)_{\epsilon}$ has a hyperbolic closed orbit near
$\gamma$. We can choose a unique positive $\epsilon_{0}$ since the
elementary poly-trajectories, are finite
 in
number. As the singularities of $X$ and $Y$ are hyperbolic, there is
no possibility of Hopf type bifurcation. So, the case of periodic
orbits emerging from singularities by the
$\varphi_\epsilon$-compactification process is excluded. As $Z \in
G_m(3)$, $\mathcal P(Z)$ does not have separatrix graphs in $S^2$
unless they are contained in $S^1$, so there is no possibility of
appearance of a periodic orbit from such a graph. So, the periodic
orbits emerging from the $\varphi_\epsilon$-compactification of
closed poly-trajectories are the only new periodic orbits of
$\mathcal P(Z)_{\epsilon}$, for $\epsilon$ small.
\end{proof}

\begin{proposition}\label{prop4.10}
Let $Z = (X,Y) \in G_m(3) \cap G_m(2)$. Then, given a transition
function $\varphi$, there is $\epsilon_{0} > 0$ such that  for  $ 0
<  \epsilon  \leq \epsilon_{0}$, $\mathcal P(Z)_{\epsilon}$ does not
have saddle connections in $S^2$ unless they are contained in $S^1$.
\end{proposition}
\begin{proof}
We claim that there is an $\epsilon_{0} > 0$  such  that  for every
$ 0 <  \epsilon  \leq \epsilon_{0}$, $\mathcal P(Z)_{\epsilon}$ does
not have  saddle connections, except on $S^1$. Indeed, as $X|_{N}$
and $Y|_{S} \in \mathcal S_m$, and $\mathcal P(Z)$ does not have
separatrix connections on $S^2$ unless they are contained in $S^1$,
the only possibilities for $\mathcal P(Z)_{\epsilon}$ to have such
separatrix connection on $S^2$, unless they are contained in $S^1$,
are as follows:
\begin{enumerate}
\item passing through points of the curve D;
\item due to the
presence of  a semi-stable periodic orbit, which could disappear
and allow a connection of two separatrices.
\end{enumerate}
Possibility 2 is discarded, since $Z \in G_m(2)$. We must analyze
possibility 1. Let $\delta$ be the minimum of the set $\{
dist(e_{i}, e_{j}): e_{i}$ is a separatrix of $\mathcal P(Z)$, and
$i \neq j\}$. Of course, $\delta > 0$, since the number of
separatrices is finite. Then, we diminish $\epsilon_{0}$ so that
the minimum distance of the separatrices for the regularized
vector field can never be less than $\frac{\delta}{2}$.
\end{proof}

Recall that  $\Sigma^r(S^2,S^1)$, $r\geq 1$, stands for structurally
stable vector fields on $S^2$ inside $\chi^r(S^2,S^1)$ (see
Definition \ref{def:02}).

\begin{theorem}\label{t4.1}
If $Z = (X,Y) \in G_m$, then, given a transition function
$\varphi$, there is $\epsilon_{0} > 0$ such that  for $0 <
\epsilon  \leq \epsilon_{0}$, then $\mathcal P(Z)_{\epsilon} \in
\Sigma^r(S^2,S^1)$, $r\geq 1$.
\end{theorem}
\begin{proof}
It follows from Propositions \ref{prop2.4}, \ref{prop4.8} and
\ref{prop4.10}.
\end{proof}

\section{Genericity }
\label{sec:04}

In this section we prove that the set $G_m$ is open and
 that
 each
discontinuous piecewise polynomial  vector field $Z$ of $\Omega_m$
can be approximated by fields  of $G_m$, i.e. we prove the
genericity of $G_m$.

\begin{theorem}\label{t5.1}
The set $G_m$ is open in $\Omega_m$.
\end{theorem}
\begin{proof}

Let $Z=(X,Y)$ be a vector field in $G_m$. It will be proved that
there is $\delta > 0$ such that if $\widehat{Z} =
(\widehat{X},\widehat{Y}) \in \Omega _m$ and $|Z-\widehat{Z}|=$ max
$\{|X-\widehat{X}|, |Y-\widehat{Y}|\} < \delta$, then $\widehat{Z}
\in G_m$. For doing  this, we have to prove that $\widehat{Z} \in
G_m(i), i = 1,2,3$.
\begin{itemize}
\item We claim that there is a $\delta_{1} > 0$ such that if
$|Z-\widehat{Z}| < \delta_{1}$, then $\widehat{Z} \in G_m(1)$.
Indeed, as $Z=(X,Y) \in G_m(1)$, we have $X|_{N}$ and $Y|_{S} \in
\mathcal S_m$, and from the openness of $\mathcal S_m$, there is
$\delta_{1} > 0$ such that if $|Z-\widehat{Z}| < \delta_{1}$, then
$\widehat{X}|_{N}$ and $\widehat{Y}|_{S} \in \mathcal S_m$. Now,
it remains to prove that if $p$ is an elementary $D$-singularity
of $\mathcal P(Z)$ and $\widehat{Z}$ is close to $Z$, then there
is a point $\widehat{p}$ near $p$ which is an elementary
$D$-singularity of $\mathcal P(\widehat{Z})$.

Let $p$ be a fold of $Z$. We can suppose that $Xf(p) = 0$, $X^2f(p)
\neq 0$ and $Yf(p) \neq 0$. As $Xf(p) = 0$ and $X^2f(p) \neq 0$, the
curve $\{Xf = 0\}$ crosses transversally the curve $D$ at the point
$p$, and, by continuity, the same occurs to the curve
$\{\widehat{X}f = 0\}$, for $\widehat{Z}$ near $Z$. This means that
there is $\widehat{p}$ near $p$ such that $\widehat{X}f(\widehat{p})
= 0$ and $\widehat{X}^2f(\widehat{p}) \neq 0$. If $\delta_{1}$ is
small enough, we can assume that it is also true that
$\widehat{Y}f(\widehat{p}) \neq 0$. So, $\widehat{p}$ is a fold of
$\widehat{Z}$. Hence, as there are no folds of $\mathcal P(Z)$ in
$S^1$, it follows that if $p$ is a fold of $\mathcal P(Z)$ then
$\widehat{p}$ near $p$ is a fold of $\mathcal P(\widehat{Z})$.

Let $p$ be a hyperbolic singularity of $F_{Z}$. We have that
\linebreak $Xf(p)Yf(p)< 0$, $det[X,Y](p) = 0$ and
$d(det[X,Y])|_{D}(p) \neq 0$. Similarly  to the fold case, the
curve $\{det[X,Y]|_{D}(p) = 0\}$ crosses transversally the curve
$D$ at the point $p$, and the same is true for $\widehat{Z}$ near
$Z$.
 So, there is a $\widehat{p}$ near $p$ such that
$\widehat{X}f(\widehat{p})\widehat{Y}f(\widehat{p}) < 0$,
$det[\widehat{X},\widehat{Y}](\widehat{p}) = 0$ and
$d(det[\widehat{X},\widehat{Y}])|_{D}(\widehat{p}) \neq 0$. This
implies that $\widehat{p}$ is a hyperbolic singular point of
$F_{\widehat{Z}}$. As $\delta_{1}$ can be chosen so that none of the
involved function change sign, and therefore $\widehat{p}$ is a
singularity of the same kind as $p$. As the D-singularities are
isolated, $\delta_{1}$ can be chosen strictly positive. We have that
$\widehat{Z}$ does not have other singularities. This is due to the
openness of the conditions that exclude this type of singularities.

Now, if $p\in S^1\cap D$ is a hyperbolic singularity of
$F_{\mathcal P(Z)}$ then, as $S^1$ is invariant by $\widehat{Z}$,
by the previous case, it follows that $p$ is a hyperbolic
singularity of $F_{\mathcal P(\widehat{Z})}$. Thus, $\widehat{Z}
\in G_m(1)$.

\item We claim that there is a $\delta_{2} > 0$ such that if
$|Z-\widehat{Z}| < \delta_{2}$, then $\widehat{Z} \in G_m(2)$.

As $Z=(X,Y) \in G_m(2)$, we have $X|_{N}$ and $Y|_{S} \in \mathcal
S_m$, and each closed poly-trajectory of $Z$ is elementary.

Let $\gamma$ be an  elementary closed poly-trajectory of type 1 of
$Z$. Associated to $\gamma$ there is a first return map $\eta$,
differentiable and such that $\eta^\prime(p) \neq 1$, for $p \in
\gamma$. This means that $p$ is a hyperbolic fixed point of the
diffeomorphism $\eta$. So, there is a number $k > 0$ such that if
$\mu$ is a diffeomorphism with $|\eta-\mu|_{1} < k$, then $\mu$ has
a hyperbolic fixed point $p_{\mu}$ near $p$. Then, it is enough to
choose $\delta_{2} > 0$ small as necessary for if $|Z-\widehat{Z}| <
\delta_{2}$, the first return map $\widehat{\eta}$ associated to
$\widehat{Z}$ satisfies $|\eta-\widehat{\eta}|_{1} < k$. So,
$\widehat{\eta}$ has a hyperbolic fixed point $\widehat{p}$ which
corresponds to an elementary closed poly-trajectory of type 1 of
$\widehat{Z}$. In the same way, if $S^1$ is a poly-trajectory of
$\mathcal P(Z)$ and so it is  of type $1$, as $S^1$ is invariant by
$\mathcal P(\widehat{Z})$, it follows that $S^1$ is  also a
poly-trajectory of $\mathcal P(\widehat{Z})$ and it is  therefore of
type $1$.

Let $\gamma$ be an  elementary closed poly-trajectory of type 3 of
$Z$. By the continuity of the functions involved, it can be shown
that  there is $\delta_{2} > 0$ such that if $|Z-\widehat{Z}| <
\delta_{2}$, $\widehat{Z}$ has an  elementary closed
poly-trajectory $\widehat{\gamma}$ of type 3 near $\gamma$.

As the number of poly-trajectories is finite, we can choose
$\delta_{2} > 0$ small enough so that  $\widehat{Z}$ has only
elementary poly-trajectories.  So, we have proved that $\widehat{Z}
\in G_m(2)$.

\item We claim that there is a $\delta_{3} > 0$ such that if
$|Z-\widehat{Z}| < \delta_{3}$, then $\widehat{Z} \in G_m(3)$.

Indeed, as $Z=(X,Y) \in G_m(3)$, we have $X|_{N}$ and $Y|_{S} \in
\mathcal S_m$ and there is $\delta_{3} > 0$ such that if
$|Z-\widehat{Z}| < \delta_{3}$, then $\widehat{X}|_{N}$ and
$\widehat{Y}|_{S} \in \mathcal S_m$. So, $\widehat{X}$ and
$\widehat{Y}$ do not have separatrix connections in $N$ and in $S$,
respectively. It remains to analyze the appearance of a connection
with at least one point in $D$. We know that $\mathcal P(Z)$ has
only a finite number of separatrices and does not have a connection
on $S^2$ unless they are contained in $S^1$. As $\mathcal
P(\widehat{Z})$ has a unique separatrix corresponding to each
separatrix of $\mathcal P(Z)$ (as follows from the uniqueness and
continuous dependence of invariant manifolds of equilibrium of
Vector Fields and fixed points of Diffeomorphisms, see \cite{HPS}),
it is easy to show that $\delta_{3}
> 0$ can be chosen so that $\mathcal P(\widehat{Z})$ does not have
separatrix connections on $S^2$ unless they are contained in $S^1$.
In this way, we have established that $\widehat{Z} \in G_m(3)$.
\end{itemize}

To finish the proof, we can take $\delta =$ min $\{\delta_{1},
\delta_{2}, \delta_{3}\}$, then if $\widehat{Z} =
(\widehat{X},\widehat{Y}) \in \Omega_m$ and $|Z-\widehat{Z}|=$ max
$\{|X-\widehat{X}|, |Y-\widehat{Y}|\} < \delta$, then $\widehat{Z}
\in G_m$. As a consequence, the set $G_m$ is open in $\Omega_m$.
\end{proof}

\begin{definition} \label{5.1}
Assume that $Z = (X,Y) \in \Omega_m$. For each pair $(\sigma, v) \in
\mathbb R^2\times \mathbb R^2$, let $Z_{\sigma,v}$ be the field $Z$
translated by $v=(v_{1},v_{2})$ and rotated by $\sigma =
(\sigma_{1},\sigma_{2})$; this means that
$$Z_{\sigma,v} = \mathcal R_{\sigma}(Z+v) = (\mathcal R_{\sigma_{1}}(X+v),\mathcal R_{\sigma_{2}}(Y+v)),$$
 where
  $$\mathcal R_{\sigma_{1}}(X+v) = \left( \begin{array}{cc} \cos \sigma_{1}& -\sin \sigma_{1}\\
\sin \sigma_{1}& \cos \sigma_{1}\\ \end{array} \right) \left(
\begin{array}{c} P_{1}+v_{1}\\ Q_{1}+v_{2}\\ \end{array}
\right).$$
\end{definition}

\begin{theorem}
The set $G_m$ is dense in $\Omega_m$.
\end{theorem}
\begin{proof}
Let $Z=(X,Y)$ be a vector field of  $\Omega_m$ with $X=(P_1,Q_1)$
and $Y=(P_2,Q_2)$. If $\mathcal P(Z)$ has singularities in $S^1$ we
can suppose that the singularities of $\mathcal P(X)$ and $\mathcal
P(Y)$ in $S^1$ are all hyperbolic and these vector fields do not
have singular points in $(\pm1,0,0)$. Otherwise, from the continuous
case (see \cite{G} and \cite{S1}), we can approximate $X$ and $Y$ by
other two vector fields with such properties. Now if some of the
points $(\pm1,0,0)\in S^1\cap D$ are not hyperbolic singularities of
$F_{\mathcal P(Z)}$, by \eqref{eq:04}, we can make these points
hyperbolic by  adding to $X$ or $Y$ a perturbation of type
\[
\epsilon x^m\frac{\partial }{\partial x}+0\frac{\partial
}{\partial y} \hspace{0.4cm}\mbox{or}\hspace{0.4cm}
0\frac{\partial }{\partial x}+\epsilon x^m\frac{\partial
}{\partial y}.
\]

Suppose that $S^1$ is a closed poly-trajectory of type $1$ of
$\mathcal P(Z)$ which is not elementary, i.e. by \eqref{eq:05}
\[
\int^{\pi}_0\left(\frac{R_{1,m}(\theta)}{A_{1,m}(\theta)}+\frac{R_{2,m}(\theta)}{A_{2,m}(\theta)}\right)d\theta=
0.
\]
Then adding to $X$ the perturbation
\[
\epsilon (x^2+y^2)^kx\frac{\partial }{\partial
x}+\epsilon(x^2+y^2)^ky\frac{\partial }{\partial y},
\]
with $m=2k+1$, it follows that the above equality becomes
\[
\int^{\pi}_0\left(\frac{R_{1,m}(\theta)}{A_{1,m}(\theta)}+\frac{R_{2,m}(\theta)}{A_{2,m}(\theta)}+\frac{\epsilon}{A_{1,m}(\theta)}\right)d\theta=
\int^{\pi}_0\frac{\epsilon}{A_{1,m}(\theta)}d\theta\neq 0.
\]
This implies that $S^1$ can be made elementary.

Notice  that if $X\in \chi_m$ then to  $\tilde{X}=\mathcal
R_{\sigma_1} (X+v)$, $(\sigma_1,v)\in \mathbb R\times \mathbb
R^2$,
\[
\begin{array}{lcl}
\tilde{A}_m(\theta) & = & \cos\sigma_1 A_m(\theta)+\sin\sigma_1
R_m(\theta), \\
\tilde{R}_m(\theta) & = & \cos\sigma_1 R_m(\theta)-\sin\sigma_1
A_m(\theta).
\end{array}
\]
Note also that we can write the condition \eqref{eq:04} as
\[
R_{1,m}(0)A_{2,m}(0)-A_{1,m}(0)R_{2,m}(0)\neq 0.
\]
Hence, as the singularities of $\mathcal P(Z)$ in
$S^1\setminus\{(\pm1,0,0)\}$ correspond by \eqref{eq:01} the points
$(\theta,0)$ such that $A_{k,m}(\theta)=0$ and they are hyperbolic
if $A_{k,m}'(\theta)R_{k,m}\neq 0$, $k=1,2$, it follows that if
$(\sigma,v)$ is small enough then $S^1$ is still either an
elementary closed poly-trajectory  of $\mathcal P(Z_{\sigma,v})$ or
all singularities of $\mathcal P(Z_{\sigma,v})$ in $S^1$ are
hyperbolic. Now, by the continuous case (see \cite{S}) and from the
proof of Theorem $25$ of \cite{MS}, we have that the set of
$(\sigma, v) \in \mathbb R^2 \times \mathbb R^2$ such that
$Z_{\sigma,v}$ has at least one non hyperbolic singularity, one non
elementary $D$-singular point, one non hyperbolic closed orbit, one
non elementary poly-trajectory or one connection of saddle
separatrizes of $\mathcal P(Z)$ in $S^2\setminus S^1$, has null
Lebesgue measure in $\mathbb R ^4$.

This finishes the proof of the theorem.
\end{proof}

\end{document}